\documentclass[12pt,twoside]{amsart}
\usepackage{amsmath}
\usepackage{amsthm}
\usepackage{amsfonts}
\usepackage{amssymb}
\usepackage{latexsym}
\usepackage{mathrsfs}
\usepackage{amsmath}
\usepackage{amsthm}
\usepackage{amsfonts}
\usepackage{amssymb}
\usepackage{latexsym}
\usepackage{geometry}
\usepackage{dsfont}
\usepackage[dvips]{graphicx}
\usepackage{color}
\usepackage[all]{xy}

\usepackage{amsthm,amsmath,amssymb}\usepackage{mathrsfs}

\date{}
\allowdisplaybreaks[4] \footskip=15pt
\renewcommand{\uppercasenonmath}[1]{}

\usepackage{graphicx,amssymb}
\usepackage[all]{xy}
\usepackage{amsmath}

\allowdisplaybreaks
\usepackage{amsthm}
\usepackage{color}

\theoremstyle{plain}
\newtheorem{theorem}{Theorem}[section]
\newtheorem{proposition}[theorem]{Proposition}

\newtheorem{corollary}[theorem]{Corollary}
\theoremstyle{definition}
\newtheorem{example}[theorem]{Example}

\theoremstyle{definition}

\theoremstyle{remark}

\newcommand{\pf}{\noindent\begin {proof}}
\newcommand{\epf}{\end{proof}}

\newcommand{\Ker}{\mbox{\rm Ker}}
\newcommand{\Ext}{\mbox{\rm Ext}}
\newcommand{\Hom}{\mbox{\rm Hom}}
\newcommand{\Tor}{\mbox{\rm Tor}}

\newcommand{\Prufer}{Pr\"{u}fer}

\newcommand{\Pj}{\mathcal{P}}
\newcommand{\C}{\mathcal{C}}
\newcommand{\g}{\mathcal{G}}
\newcommand{\G}{\mathcal{G}}
\newcommand{\St}{\mathcal{S}}
\newcommand{\FPR}{\mathcal{FPR}}

\newcommand{\T}{\mathcal{T}}
\newcommand{\Ss}{\mathcal{S}}

\newcommand{\Prod}{\mathrm{Prod}}
\newcommand{\Add}{\mathrm{Add}}
\newcommand{\Id}{\mathrm{Id}}

\newcommand{\A}{\mathcal{A}}
\newcommand{\B}{\mathcal{B}}

\newcommand{\Proj}{\mathcal{P}}

\def\GV{{\rm GV}}

\def\Hom{{\rm Hom}}
\def\Ext{{\rm Ext}}
\def\Tor{{\rm Tor}}

\def\Ker{{\rm Ker}}

\def\Ann{{\rm Ann}}

\def\GV{{\rm GV}}
\def\Max{{\rm Max}}
\def\DW{{\rm DW}}
\def\Add{{\rm Add}}
\def\gld{{\rm gld}}
\def\fPD{{\rm fPD}}
\def\FPD{{\rm FPD}}
\def\VAss{{\rm VAss}}
\def\Spec{{\rm Spec}}

\def\pd{{\rm pd}}
\def\V{{\rm V}}
\def\m{{\frak m}}
\def\grade{{\rm grade}}

\def\Mod{{\rm Mod}}
\def\p{{\frak p}}

\def\DQ{{\rm DQ}}
\def\LPVD{{\rm LPVD}}
\def\id{{\rm id}}

\begin{document}
\begin{center}
{\large  \bf  The small finitistic dimensions of commutative rings}

\vspace{0.5cm}   Xiaolei Zhang$^{a}$, Fanggui Wang$^b$\\

{\footnotesize a.\ School of Mathematics and Statistics, Shandong University of Technology, Zibo 255000, China\\
b.\ School of Mathematical Sciences, Sichuan Normal University, Chengdu 610068, China\\
E-mail: zxlrghj@163.com\\}
\end{center}

\bigskip
\centerline { \bf  Abstract} Let $R$ be a commutative ring with identity.  The small finitistic dimension $\fPD(R)$ of $R$ is defined  to be the supremum of projective dimensions of $R$-modules with finite projective resolutions.  In this paper, we characterize a ring $R$ with $\fPD(R)\leq n$ using  finitely generated semi-regular ideals, tilting modules, cotilting modules of cofinite type and  vaguely associated  prime ideals. As an application,  we obtain that  if $R$ is a Noetherian ring, then $\fPD(R)= \sup\{\grade(\m,R)|\m\in \Max(R)\}$  where $\grade(\m,R)$ is the grade of $\m$ on $R$ . We also  show that a ring $R$ satisfies $\fPD(R)\leq 1$ if and only if $R$ is a $\DW$ ring. As applications, we show that the small finitistic dimensions of strong \Prufer\ rings and $\LPVD$s are at most one. Moreover, for any given $n\in \mathbb{N}$, we obtain a total ring of quotients $R$ satisfying $\fPD(R)=n$.
\bigskip
\leftskip10truemm \rightskip10truemm \noindent
\vbox to 0.3cm{}

\leftskip0truemm \rightskip0truemm
\bigskip
\section{introduction}
Throughout this paper, $R$ is a commutative ring with identity and Mod-$R$ is the category of all $R$-modules.
Let $M$ be an $R$-module, we use $\pd_RM$  to denote the projective dimension of $M$ over $R$. The global dimension of $R$, denoted by $\gld(R)$, is defined to be the supremum of the projective dimensions of all $R$-modules. 
It is well-known that a Noetherian local ring $R$ has finite global dimension if and only if $R$ is a regular local ring.
So Noetherian local rings with nonzero zero divisors have infinite global dimensions. This motivates the  definition of the so called finitistic dimension of a ring $R$. The big finitistic dimension of $R$, denoted by $\FPD(R)$, is defined to be the supremum of the projective dimensions of $R$-modules $M$ with finite projective dimensions. Bass \cite{B60} proved that a ring $R$ is perfect if and only if $\FPD(R)=0$.  For a  Noetherian ring $R$, Raynaud and Gruson \cite{RG71} showed that the big finitistic dimension  $\FPD(R)$ of $R$ coincides with the Krull dimension $K.\dim(R)$ of $R$.  

For a ring $R$, we denote mod-$R$ to be the full subcategory of all $R$-modules that have finitely generated projective resolutions, i.e. the $R$-modules $M$  fitting into an exact sequence
$$\dots\rightarrow P_m\rightarrow P_{m-1}\rightarrow \dots\rightarrow P_1\rightarrow P_0\rightarrow M\rightarrow 0$$ where each $P_i$ is a finitely generated projective $R$-module. Let $M$ be an $R$-module. Then $M$ is said to have a finite projective resolution, denoted by $M\in\FPR$, if  there exist an integer $n$ and an exact sequence
$$0\rightarrow P_n\rightarrow P_{n-1}\rightarrow \dots\rightarrow P_1\rightarrow P_0\rightarrow M\rightarrow 0$$
with each $P_i$  finitely generated projective. We denote $\Proj^{\leq n}$  to be the class of $R$-modules with projective dimensions at most $n$ in  $\FPR$. The small finitistic dimension of $R$, denoted by $\fPD(R)$, is defined to be the supremum of the projective dimensions of $R$-modules in $\FPR$. Clearly, $\fPD(R)\leq n$ if and only if $\FPR=\Proj^{\leq n}$, and $\fPD(R)\leq \FPD(R)\leq \gld(R)$ for any ring $R$. 
In case $R$ is commutative Noetherian local, Auslander and Buchweitz showed that the small finitistic dimension $\fPD(R)$ of $R$  coincides with the depth of $R$ (see \cite{AB58}).  In 2020, Wang et al. \cite{fkxs20} proved that a ring $R$ is a $\DQ$ ring if and only if $\fPD(R)=0$. They also showed that if $\fPD(R)\leq 1$, then $R$ is a $\DW$ ring and then gave  examples of \Prufer\ rings with small finitistic dimensions larger than 1 getting a negative answer to the open question raised by  Cahen et al. \cite{CFFG14}.

In the present paper, we characterize the small finitistic dimension of a ring $R$ using some special finitely generated semi-regular ideals,  tilting modules,  cotilting modules of cofinite type and the vaguely associated  prime ideals of $R$ (see Theorem \ref{main}).  As an application,  we obtain that  if $R$ is a Noetherian ring, then $\fPD(R)= \sup\{\grade(\m,R)|\m\in \Max(R)\}$ where $\grade(\m,R)$ is  the grade of $\m$ on $R$ (see Proposition \ref{nof}). For little small finitistic dimensions, we rediscover the rings with $\fPD(R)=0$ (see Corollary \ref{DQ}) and then show that a ring $R$ is a $\DW$ ring if and only if $\fPD(R)\leq 1$ (see Corollary \ref{DW}).  As applications, we show that the small finitistic dimensions of strong \Prufer\ rings and $\LPVD$s are at most one (see Corollary \ref{spf} and Corollary \ref{lpvd}). Moreover, for any given $n\in \mathbb{N}$, we obtain examples of total rings of quotients with small finitistic dimensions $n$ which gives a deeper understanding of Cahen et al.'s open question \cite{CFFG14} (see Example \ref{exa-fpd-n}).

\section{Preliminaries}

In this section, we  recall the classifications of infinitely generated tilting modules over commutative rings developed by Hrbek and \v{S}\v{t}ov\'{i}\v{c}ek (see \cite{H16, HS19}). This theory is essential to our studies of the small finitistic dimensions of commutative rings.

Let $R$ be a commutative ring, $M$ an $R$-module.   $\Omega^{-i}(M)$ is denoted to be the $i$-th minimal cosyzygy of $M$. For a class $\St$ of $R$-modules, we always use the following notations:
$$\St^{\bot}=\{M\in \Mod\mbox{-}R|\Ext^n_R(S,M)=0, \forall S\in\St,\ \forall n\geq 1\},$$
$$^{\bot}\St=\{M\in \Mod\mbox{-}R|\Ext^n_R(M,S)=0, \forall S\in\St,\  \forall n\geq 1\},$$
$$\St^{\top}=\{M\in \Mod\mbox{-}R|\Tor_n^R(S,M)=0, \forall S\in\St,\  \forall n\geq 1\}.$$

Following from \cite{S75}, a filter $\G$ of ideals of $R$ is called a Gabriel topology provided that:
\begin{enumerate}
    \item if $I\in\G$ and $t\in R$, then $(I:t):=\{r\in R|tr\in I\}\in\G$,
    \item if $J$ is an ideal and $I\in\G$ such that $(J:t)\in \G$ for any $t\in I$, then $J\in\G$.
\end{enumerate}
A Gabriel topology is finitely generated if it has a basis of finitely generated ideals. We denote by $\G^f$ the set of all finitely generated ideals in $\G$. Obviously, a Gabriel topology is finitely generated if and only if it can be generated by $\G^f$.

Given a class $\C$ of $R$-modules, let SubLim($\C$) denote the smallest subclass of $\Mod$-$R$ containing $\C$ closed under direct limits and submodules. A prime ideal $\p$ of $R$ is called vaguely associated to an $R$-module $M$ if $R/\p$ is contained in SubLim($\{M\}$). Denote the spectrum of all vaguely associated primes of $M$ by $\VAss(M)$. The set of vaguely associated prime ideals is a general version of that of associated prime ideals over Noetherian rings (see \cite{H16}).

An $R$-module $T$ is said to be $n$-tilting for some $n\geq 0$, provided that it satisfies the following conditions:
\begin{enumerate}
    \item $\pd_RT\leq n$,
    \item $\Ext_R^i(T,T^{(\kappa)})=0$ for all $i>0$ and all cardinals $\kappa$,
    \item there are $r\geq 0$ and a long exact sequence
    $$0\rightarrow R\rightarrow T_0\rightarrow T_1\rightarrow \dots\rightarrow T_n\rightarrow 0$$
    where $T_i\in\Add(T)$ for all $0\leq i\leq r$, where $\Add(T)$ denotes the class of all direct summands of arbitrary direct sums of copies of $T$.
\end{enumerate}
A cotorsion pair $(\A,\B)$ is said to be induced by an $R$-module $M$ provided that  $\B=M^{\bot}$. A cotorsion pair $(\A,\B)$ induced by an $n$-tilting module is said to be an $n$-tilting cotorsion pair, and then the class $\B$ is said to be the $n$-tilting class. We suppress the index $n$- in the notation if we do not desire to specify the dimension bound on $T$. Two tilting modules are said to be equivalent provided that they induce the same tilting class. An $R$-module $M$ is called super finitely presented if it admits a projective resolution consisting of finitely generated projective modules. The following result, which can be implied by \cite[Theorem 4.2]{BJ13} and \cite[Theorem 13.26]{GT12}, shows that all tilting classes are induced by super finitely presented modules with bounded projective dimensions.

\begin{theorem}\label{f-tilt}
A class $\B$ of $R$-modules is $n$-tilting if and only if there is a set $\St$ of super finitely presented modules of projective dimension at most $n$ such that $\B=\St^{\bot}$.
\end{theorem}

An $R$-module $C$ is said to be $n$-cotilting for some $n\geq 0$, provided that it satisfies the following conditions:
\begin{enumerate}
    \item $\id_RC\leq n$,
    \item $\Ext_R^i(C^{\kappa},C)=0$ for all $i>0$ and all cardinals $\kappa$,
    \item there are $r\geq 0$ and a long exact sequence
    $$0\rightarrow C_r\rightarrow\dots  \rightarrow C_1\rightarrow C_0\rightarrow  W\rightarrow 0$$
    where $W$ is an injective cogenerator in $R$-Mod and $C_i\in\Prod(C)$ for all $0\leq i\leq r$, where $\Prod(C)$ denotes the class of all direct summands of arbitrary direct products of copies of $C$.
\end{enumerate}
A cotorsion pair $(\A,\B)$ is said to be co-induced by an $R$-module $C$ provided that $\A=^{\bot}C$.  A cotorsion pair $(\A,\B)$ co-induced by an $n$-cotilting module is said to be an  $n$-cotilting cotorsion pair, and then the class $\A$ is said to be the $n$-cotilting class.  A class $\C$ of $R$-modules is of  cofinite type if there exists $n<\infty$ and a set $\St$ of super finitely presented modules of projective dimensions at most $n$ such that $\C=\Ss^{\top}$. An $R$-module $C$ is of cofinite type if the class $^{\bot}C$ of $R$-modules is of cofinite type.

Recall that a subcategory $\St$ of mod-$R$ is said to be resolving, if all finitely generated projective modules are contained in $\St$, and $\St$ is closed under extensions, direct summands, and and $A\in\St$ where there is an exact sequence $0\rightarrow A\rightarrow B\rightarrow C\rightarrow 0$ with $B,C\in\St$.

Over commutative rings, Hrbek and \v{S}\v{t}ov\'{i}\v{c}ek \cite{HS19} classified the tilting classes in terms of sequences of Gabriel topologies, cotilting classes of cofinite type and resolving subcategories of mod-$R$.

\begin{theorem}\cite[Theorem 6.2]{HS19}\label{corr}
Let $R$ be a commutative ring and $n\geq 0$. There are $1$-$1$ correspondences between the following collections:
\begin{enumerate}
    \item sequences $(\g_0,\g_1,\dots,\g_{n-1})$ of Gabriel topologies of finite type satisfying:\\
           $($a$)$ $\g_0\subseteq\g_1\subseteq\dots\subseteq\g_{n-1}$,\\
           $($b$)$ $\Ext_R^j(R/I,R)=0$ for all $I\in \g_i$, all $i=0,1,\dots,n-1$, all $j=0,1,\dots,i$,
    \item $n$-cotilting classes $\C$ in Mod-$R$ of cofinite type,
    \item $n$-tilting classes $\T$ in Mod-$R$,
    \item resolving subcategories of mod-$R$ consisting of modules of projective dimension at most $n$.

\end{enumerate}
The correspondences are given as follows:\\

    $(1)\rightarrow (2)$:  $(\g_0,\g_1,\dots,\g_{n-1})\mapsto \displaystyle\bigcap_{i=0}^{n-1}\displaystyle\bigcap_{I\in \g_i^{f}}\Ker\Ext_R^i(R/I,-)=\displaystyle\bigcap_{i=0}^{n-1}\displaystyle\bigcap_{I\in \g_i^{f}}(S_{I,i+1})^{\top}$

     $(1)\rightarrow (3)$: $(\g_0,\g_1,\dots,\g_{n-1})\mapsto \displaystyle\bigcap_{i=0}^{n-1}\displaystyle\bigcap_{I\in \g_i^{f}}\Ker\Tor^R_i(R/I,-)=\displaystyle\bigcap_{i=0}^{n-1}\displaystyle\bigcap_{I\in \g_i^{f}}(S_{I,i+1})^{\bot}$.\\

    $(1)\rightarrow (4)$: $(\g_0,\g_1,\dots,\g_{n-1})\mapsto \{M\in$ mod-$R| M$ is isomorphic to a summand of a

     \qquad\qquad finitely $\{R\}\cup\{S_{I,i+1}|I\in \g_i^{f}, i<n\}$-filtered module$\}$.\\
\end{theorem}

\section{main results}

In this section, we give the main results of this paper and some applications.
\begin{theorem}\label{main}
Let $R$ be a commutative ring and $n\geq 0$. The following conditions are equivalent:
\begin{enumerate}
    \item   $\fPD(R)\leq n$,
    \item   any  tilting module is $n$-tilting,
     \item   any  cotilting module of cofinite type is $n$-cotilting,
     \item  any  finitely generated ideal $I$ that satisfies $\Ext_R^i(R/I,R)=0$ for each $i=0,\dots,n$ is $R$,
     \item  any  finitely generated ideal $I$ that satisfies $\VAss(\Omega^{-i}(R))\cap \V(I)=0$ for each $i=0,\dots,n$ is $R$.
\end{enumerate}
\end{theorem}

\begin{proof}
$(1)\Rightarrow (2)$:  Let $T$ be an $m$-tilting module with $m\geq n$ and $(\A,\B)$ be the cotorsion pair induced by $T$.  By \cite[Theorem 13.46]{GT12}, there is an $\A^{<\infty}$-filtered $m$-tilting module $T_{fin}$ such that $T_{fin}$ is equivalent to $T$, where  $\A^{<\infty}$ denotes the subclass mod-$R\cap \A$ of $\A$. By \cite[Lemma 13.10]{GT12}, every $R$-module in $\A$ has a projective dimension at most $m$.  It follows that $T_{fin}$ is $\Proj^{\leq m}$-filtered. Since $\fPD(R)\leq n$  and $m\geq n$, we have $\Proj^{\leq m}=\Proj^{\leq n}$. Thus $T_{fin}$ is $\Proj^{\leq n}$-filtered. By Auslander Lemma, $\pd_RT=\pd_RT_{fin}\leq n$ and thus $T$ is $n$-tilting.

$(2)\Rightarrow (1)$: Let $M\in \FPR$ and $(\A,\B)$ be the cotorsion pair induced by $M$. By Theorem \ref{f-tilt}, $(\A,\B)$ is a tilting cotorsion pair induced by a tilting module $T$. By (2), $T$ is $n$-tilting. Since $M\in \A$, $M$ has a projective dimension at most $n$ (see \cite[Lemma 13.10]{GT12}). Since $M\in \FPR$, we have $M\in \Proj^{\leq n}$.

$(2)\Leftrightarrow (3)$: It follows by \cite[Theorem 15.18]{GT12}.

$(1)\Rightarrow (4)$: Let $I=\langle a_1,\dots,a_m\rangle$ be a finitely generated ideal satisfying $\Ext_R^i(R/I,R)=0$ for all $i=0,\dots,n$. Let $K_{\cdot}(I)$ be the Koszul complex induced by $\{a_1,\dots,a_m\}$ as follows:
$$\dots \xrightarrow{d_{n+2}} F_{n+1}\xrightarrow{d_{n+1}} F_n\xrightarrow{d_{n}} \dots \xrightarrow{d_{2}} F_1 \xrightarrow{d_{1}} F_0 \rightarrow 0.$$
Note that $F_0=R$, $F_1=R^m$ and $d_1(x_1,\dots,x_m)=\sum_{i=1}^ma_ix_i$. Denote $S_{I,n+1}$  to be the cokernel of the map $d^{\ast}_{n+1}:=\Hom_R(d_{n+1}, R)$.  Since $\Ext_R^i(R/I,R)=0$ for each $i=0,\dots,n$, by \cite[Proposition 3.9]{HS19}, we have the following exact sequence:

$$0\rightarrow  F^{\ast}_{0}\xrightarrow{d^{\ast}_{1}} F^{\ast}_1\xrightarrow{d^{\ast}_{2}} \dots \xrightarrow{d^{\ast}_{n}} F^{\ast}_n\xrightarrow{d^{\ast}_{n+1}} F^{\ast}_{n+1} \rightarrow S_{I,n+1}\rightarrow  0 $$
where $F^{\ast}_i=\Hom_R(F_i, R)$ and $d^{\ast}_{i}=\Hom_R(d_{i}, R)$ for each $i$. Thus $\pd_R(S_{I,n+1})\leq n+1$. Since  $\fPD(R)\leq n$, we have $\pd_RS_{I,n+1}\leq n$. Thus  $d^{\ast}_{1}: R^{\ast}\rightarrow (R^m)^{\ast}$ is a split monomorphism. Since the map $d_{1}: R^m\rightarrow R$ satisfies $d_{1}(x_1,\dots,x_m)=\displaystyle\sum_{i=1}^m x_ia_i$, we have $d^{\ast}_{1}(r)=(ra_1,\dots,ra_m)$ under the natural isomorphisms $R^{\ast}\cong R$ and $(R^m)^{\ast}\cong R^m$. Then there exists a homomorphism $g:R^m\rightarrow R$ such that $gd^{\ast}_{1}=\Id_R$. Thus $\displaystyle\sum_{i=1}^m g(e_i)a_i=1$ where $\{e_1,\dots,e_m\}$ is the standard basis for $R^n$. Consequently, we have $I=R$.

$(4)\Rightarrow (2)$: Let $\T$ be a tilting class induced by an $m$-tilting module $T$ and  $(\g_0,\g_1,\dots,\g_{m-1})$ the corresponding sequence of Gabriel topologies (see Theorem \ref{corr}). If $m\leq n$, $(2)$ holds  obviously. Suppose $m> n$.  Let $I$ be a finitely generated ideal in $\g_{n}$. We have $\Ext_R^i(R/I,R)=0$ for $i=0,\dots,n$, and then $I=R$ and thus $\g_{m-1}=\dots=\g_{n}=\{R\}$. Cosequently, the corresponding tilting class $\T=\displaystyle\bigcap_{i=0}^{m-1}\displaystyle\bigcap_{I\in \g_i^{f}}\Ker\Tor^R_i(R/I,-)=\displaystyle\bigcap_{i=0}^{n-1}\displaystyle\bigcap_{I\in \g_i^{f}}\Ker\Tor^R_i(R/I,-)$ is induced by $(\g_0,\g_1,\dots,\g_{n-1})$ and thus an $n$-tilting class by Theorem \ref{corr}. Consequently, $T$ is $n$-tilting.

$(4)\Leftrightarrow (5)$: It follows from \cite[Proposition 3.13]{HS19}.
\end{proof}

Let $R$ be a Noetherian ring and $I$  a finitely generated proper ideal of $R$. Following \cite{BH93}, the common length of the maximal $R$-sequences in $I$, denoted by $\grade(I,R)$, is said to be the grade of $I$ on $R$. A well-known result is that $\grade(I,R)=\min\{i\geq 0|\Ext_R^i(R/I,R)\not=0\}$. If $(R,\m)$ is a Noetherian local ring, then $\grade(\m,R)$  is said to be the depth of $R$.  Auslander and Buchsbaum \cite{AB58} showed that the small finitistic dimension  $\fPD(R)$ of $R$  is equal to the depth of $R$. Now we extend this classical result to general Noetherian rings.

\begin{proposition}\label{nof}
Let $R$ be a Noetherian ring. Then $\fPD(R)= \sup\{\grade(\m,R)|\m\in \Max(R)\}$.
\end{proposition}

\begin{proof}  Suppose $\fPD(R)$ is infinite. By Theorem \ref{main}, for any $n\geq 0$, there exists a proper ideal $I$ such that $\Ext_R^i(R/I,R)=0$ for each $i=0,\dots,n$. Let $\m$ be the maximal ideal containing $I$. Then $\grade(\m,R)\geq \grade(I,R)\geq n$. Thus $\sup\{\grade(\m,R)|\m\in \Max(R)\}$ is infinite. Suppose $\fPD(R)=n$ for a non-negative integer $n$.  Then, by Theorem \ref{main}, there exists a proper ideal $I$ of $R$ such that $\Ext_R^i(R/I,R)=0$ for each $i=0,\dots,n-1$. Let $\m$ be the maximal ideal containing $I$. Then  we have
$\grade(\m,R)\geq \grade(I,R)\geq n$. If there is a maximal ideal $\m'$ such that $\grade(\m',R)>n$, then $\Ext_R^i(R/\m',R)=0$ for all $i=0,\dots,n$. Thus $\fPD(R)>n$ by Theorem \ref{main} which is a contradiction.
\end{proof}

A Noetherian local ring $(R,\m)$ is said to be a Cohen-Macaulay ring provided that $\grade(\m,R)=K.\dim(R)$ where $K.\dim(R)$ denotes the Krull dimension of $R$.  In general, a Noetherian ring  $R$ is said to be a Cohen-Macaulay ring provided that $R_{\m}$ is Cohen-Macaulay for each $\m\in \Max(R)$. For a  Noetherian ring $R$, Raynaud and Gruson \cite{RG71} proved that the big finitistic dimension  $\FPD(R)$ of $R$ coincides with the Krull dimension $K.\dim(R)$ of $R$.  Thus by Proposition \ref{nof}, we have the following result.
\begin{corollary}\label{cm}
Let $R$ be a Cohen-Macaulay ring. Then  $\fPD(R)=\FPD(R)$. In this case, $\fPD(R)=\FPD(R)=K.\dim(R)$.
\end{corollary}

If $R$ is a Noetherian local ring, the converse of Corollary \ref{cm} always holds. However, the converse is not always true for general Noetherian rings.
\begin{example}
Let $R=\mathbb{Z}[pX, X^2, X^3]$ where $p$ is a fixed prime. It is easy to check that $R$ is a Noetherian ring with  $K.\dim(R)=2$, and thus $\FPD(R)=2$ (see \cite[Theorem 3.2.6]{RG71}). We denote $\m_q=\langle q, pX, X^2, X^3\rangle$ for each prime $q$, and then $\m_q$ is a maximal ideal of $R$. One can show that $\grade(\m_q,R)=1$ if $q=p$, and $\grade(\m_q,R)=2$ if $q\not=p$. Thus $\fPD(R)=2$ by Proposition \ref{nof}.  Note that  $K.\dim(R_{\m_q})=2$ for each prime $q$. It follows that $R_{\m_p}$ is not Cohen-Macaulay and thus $R$ is not a Cohen-Macaulay ring.
\end{example}

The following example shows that the small finitistic dimensions of Noetherian rings can be infinite.
\begin{example}
Let $R$ be the Nagata's bad Noetherian ring given in \cite[Appendix,
Example 1]{N62}. Namely, let $D=k[x_1,\dots,x_n,\dots]$ be the polynomial ring with countably infinite variables over a field $k$. Let $\p_i=\langle x_{2^{i-1}},x_{2^{i-1}+1},\dots, x_{2^i-1}\rangle$ for $i\geq 1$ which is a prime ideal of $D$. Let $S$ be the multiplicative subset$$S=\bigcap_{i\geq 1}(R-\p_i).$$ Consider the ring $R=D_S$. By \cite[Section 02JC]{Stack}, $R$ is a Cohen-Macaulay ring with infinite Krull dimension. So $\fPD(R)=\infty$ by Corollary \ref{cm}.
\end{example}

Recall from \cite{L93},  an ideal $I$ of $R$ is said to be dense if $\Ann(I):=\{r\in R|Ir=0\}=0$, and  semi-regular if it contains a finitely generated dense sub-ideal.  Recall from \cite{wzcc20}, a ring $R$ is said to be a $\DQ$ ring provided that the only finitely generated semi-regular ideal of $R$ is $R$ itself. Note that a finitely generated ideal $I$ is semi-regular if and only if $\Hom_R(R/I,R)=0$. Thus a ring $R$ is $\DQ$ if and only if  any finitely generated ideal $I$ satisfying $\Hom_R(R/I,R)=0$ is $R$. Thus the following result can be deduced from Theorem \ref{main} in the case of $n=0$ .

 \begin{corollary}\cite[Proposition 2.2]{fkxs20}\label{DQ}
Let $R$ be a commutative ring. Then $\fPD(R)=0$ if and only if $R$ is a $\DQ$ ring.
\end{corollary}

Following from \cite{fk16},  a finitely generated ideal $J$ of $R$ is said to be a $\GV$-ideal if and only if $\Hom_R(R/J,R)=\Ext_R^1(R/J,R)=0$. The set of all $\GV$-ideals of $R$ is denoted by $\GV(R)$. A ring $R$ is said to be a $\DW$ ring provided that  $\GV(R)=\{R\}$. Examples of $\DW$ rings contain rings of Krull dimension equal to 0, integer domains of  Krull dimension at most $1$, \Prufer\ domains and so on.   Wang et al. \cite[Theorem 3.2.]{fkxs20} showed that a commutative ring $R$ with  $\fPD(R)\leq 1$ is a $\DW$ ring. Obviously, Theorem \ref{main} means that it is actually an equivalence.

\begin{corollary}\label{DW}
Let $R$ be a commutative ring. Then $\fPD(R)\leq 1$ if and only if $R$ is a $\DW$ ring.
\end{corollary}

Recall that a commutative ring $R$ is said to be a strong \Prufer\ ring provided that every finitely generated semi-regular ideal is locally principal (see \cite{L93}). Recently, Wang et al. showed that the small finitistic dimensions of a connected strong \Prufer\ ring is at most $1$ in \cite[Theorem 2.4]{fqz21}. Now we extend this result to all strong \Prufer\ rings.

\begin{corollary}\label{spf}
Let $R$ be a strong \Prufer\ ring. Then $R$ is a $\DW$ ring. Consequently, $\fPD(R)\leq 1$.
\end{corollary}
\begin{proof} Let $J\in \GV(R)$. Then $J$ is a finitely generated semi-regular ideal of $R$. Since $R$ is strong \Prufer, for any $\p\in \Spec(R)$,  there exists a regular element $ \frac{a}{b}\in J_{\p}$ such that $J_{\p}=\langle \frac{a}{b}\rangle$.  Therefore $J_{\p}$ is free over $R_{\p}$ for any $\p\in \Spec(R)$. Thus $J$ is flat, we have $J=R$ (see \cite[Theorem 6.7.24]{fk16} and  \cite[Exercise 6.10(1)]{fk16} ).
\end{proof}

Recall from \cite{HH78} that a prime ideal $\p$ of  an integral domain $R$ is said to be strongly prime  if
whenever $xy\in \p$, where $x, y \in K$ and $K$ is the quotient field of $R$, we have either $x\in \p$ or  $y\in \p$. An integral domain $R$ is called a pseudo-valuation domain (PVD for short)  if every prime ideal of $R$ is strongly prime. Wang et al.\cite{fkx20} computed the finitistic weak dimensions of PVDs using pullbacks of commutative rings.  Recall from  \cite{DF82} that an integral domain $R$ is called a locally pseudo-valuation domain (LPVD for short) if $R_{\m}$ is a PVD for each maximal ideal $\m$ of $R$. Certainly, PVDs are LPVDs. Recently, Xie \cite{x21} pointed out that LPVDs are $\DW$ rings (see also \cite[Theorem 2.9]{M05}, \cite[Theorem 11.8.5]{fk16} and \cite[Lemma 3.1.3]{zhou17}). Thus we have the following result.

\begin{corollary}\label{lpvd}
Let $R$ be an \LPVD, then  $\fPD(R)\leq 1$.
\end{corollary}

Recall from \cite{G69} that a commutative ring $R$ is said to be a \Prufer\ ring provided that every finitely generated regular ideal is invertible. Obviously, every total ring of quotients (i.e. any regular element is invertible) is \Prufer. In \cite[Problem 1]{CFFG14}, Cahen et al. posed the following two open questions:
\begin{itemize}
\item {\bf Problem 1a:} Let $R$ be a \Prufer\ ring. Is $\fPD(R)\leq 1$?
\item {\bf Problem 1b:} Let $R$ be a total ring of quotients. Is $\fPD(R)=0$?
\end{itemize}
Very recently, Wang et al. \cite{wzcc20,fkxs20} obtained a total ring of quotients $R$ but not a \DW\ ring, and thus $\fPD(R)>1$ getting a negative answer to these two open questions. The next example shows that, for any $n\in \mathbb{N}$, there exists a  total ring of quotients $R$ satisfying $\fPD(R)=n$.

\begin{example}\label{exa-fpd-n} Let $D=k[x_1,\dots,x_n]$ be a polynomial rings with $n$ variables over a field $k$.  Set $\m=\langle x_1,\dots,x_n\rangle$ be a maximal ideal of $D$ and $\Pj=\Max(R)-\{\m\}$. Define $R=D(+)B$ to be the idealization constructed prior to \cite[Theorem 11]{L93}, where $B=\bigoplus\limits_{\p\in\Pj}D/\p$. Then $R$ is a total ring of quotients by \cite[Theorem 11(a)]{L93}.  By \cite[Theorem 11(c)]{L93}, the set of all semi-regular ideals of $R$ is $\{J(+)B\mid J$ is an ideal of $D$ and $\sqrt{J}=\m\}$. Let $J$ be an ideal of $D$ satisfying $\sqrt{J}=\m$ and $I=J(+)B$. Then, for any $i\geq 0$,  $\Ext_R^i(R/I,R)\cong \Ext_R^i(D/J,R)\cong \Ext_R^i(D/J\otimes_DR,R)$ since $J+\p=R$ for any $\p\in \Pj$. By localizing at all maximal ideals, it is easy to verify $\Tor_i^D(D/J,R)=0$ for any $i\geq 1$.
By \cite[Chapter VI, Proposition 4.1.3]{CE99}, we have $\Ext_R^i(D/J\otimes_DR,R)\cong \Ext_D^i(D/J,R)=\Ext_D^i(D/J,D(+)B)$  for any $i\geq 1$. By localizing at all maximal ideals, one can also verify $\Ext^i_D(D/J,B)=0$  for any $i\geq 0$ (note that $\Hom_R(R/I,R)\cong \Hom_D(D/J,D)=0$). Thus we have $\Ext_R^i(R/I,R)\cong \Ext_D^i(D/J,D)$  for any $i\geq 0$. By Theorem \ref{main} and \cite[Proposition 1.2.10]{BH93}, we have $\fPD(R)=\grade(J,D)=\grade(\m,D)=K.\dim(D_{\m})=n$ since $D_{\m}$ is a regular local ring.
\end{example}


\bigskip
\begin{thebibliography}{99}

\bibitem{AB58} M. Auslander, D. Buchsbaum, Homological dimension in Noetherian rings, II, Trans. Amer. Math. Soc. 88 (1958) 194-206.

\bibitem{B60} H. Bass, Finitistic dimension and a homological generalization of semiprimary rings, Trans. Amer. Math. Soc. 95 (1960) 466-488.

\bibitem{bg07}  S. Bazzoni,  S. Glaz, Gaussian properties of total rings of quotients,  J. Algebra  310 (2007) 180-193.

\bibitem{BJ13} S. Bazzoni, J. \v{S}\v{t}ov\'{i}\v{c}ek, All tilting modules are of finite type, Proc. Amer. Math. Soc.135 (2007) 3771-3781.

\bibitem{BH93} W. Bruns, J. Herzog, Cohen-Macaular Rings, Cambridge Studies in Advanced Math, Vol 39,  Cambridge University Press, Cambridge, 1993.

\bibitem{CFFG14}  P. J. Cahen,  M. Fontana, S. Frisch, S. Glaz, Open problems in commutative ring theory, In: M. Fontana,  S. Frisch, S. Glaz, eds. Commutative Algebra. New York: Springer, 2014.

\bibitem{CE99} H. Cartan, S. Eilenberg, Homological Algebra, Princeton Univ. Press, Princeton, 1956.

\bibitem{Stack} A.J. de Jong  et al., The Stacks project, https://stacks.math.columbia.edu/.

\bibitem{DF82} D. E. Dobbs, M. Fontana, Locally pseudo-valuation domains, Ann. Mat. Pura Appl. 134(4) (1982) 147-168.

\bibitem{GT12} R. G\"{o}bel, J. Trlifaj, Approximations and Endomorphism Algebras of Modules. Volume 1. Approximations, Second revised and extended edition, De Gruyter Exp. Math. 41, Walter de Gruyter, Berlin, 2012.

\bibitem{G69} M. Griffin, Pr\"{u}fer rings with zero divisors, J. Reine. Angew. Math. 239 (1969) 55-67.

\bibitem{HH78} J. R. Hedstrom, E. G. Houston, Pseudo-valuation domains, Pacific J. Math. 75 (1978) 137-147.

\bibitem{H69} M. Hochster, Prime ideal structure in commutative rings, Trans. Amer. Math. Soc. 142 (1969) 43-60.

\bibitem{H16} M. Hrbek, One-tilting classes and modules over commutative rings, J. Algebra 462 (2016) 1-22.

\bibitem{HS19} M. Hrbek,   J. \v{S}\v{t}ov\'{i}\v{c}ek, Tilting classes over commutative rings, Forum Math. 32(1) (2020) (in press).

\bibitem{L93} T. Lucas, Strong Pr\"{u}fer rings and the ring of finite fractions, J. Pure Appl. Algebra, 84 (1993) 59-71.

\bibitem{M05} A. Mimouni, Integral domains in which each ideal is a $w$-ideal, Comm. Algebra, 2005, 33: 1345-1355.

\bibitem{N62} M.  Nagata,  Local Rings. New York, NY, USA: Interscience, 1962.

\bibitem{RG71} M. Raynaud, L. Gruson, Crit\`{e}res de platitude et de projectivit\'{e}, Invent. Math. 13 (1971) 1-89.

\bibitem{S75} B. Stenstr\"{o}m, Rings of Quotients. Springer-Verlag, New York-Heidelberg, 1975.


\bibitem{fk16} F. G. Wang,  H. Kim,   Foundations of Commutative Rings and Their Modules, Singapore, Springer, 2016.

\bibitem{fkx20}  F. G. Wang, H. Kim, T. Xiong,  Finitistic weak dimensions of pullbacks, J. Pure Appl. Algebra 224 (2020), no. 6, 106274, 12 pp.

\bibitem{fqz21} F. G. Wang, L. Qiao, D. C. Zhou,   A homological characterization of strong \Prufer\ rings, Acta Math. Sinica (Chin. Ser.) 64(2021), no. 2, 311-316.

\bibitem{wzcc20}  F. G. Wang, D. C. Zhou, D. Chen,  Module-theoretic characterizations of the ring of finite fractions of a commutative ring, J. Commut. Algebra,  to appear. https://projecteuclid.org/euclid.jca/1589335712.

\bibitem{fkxs20} F. G. Wang, D. C. Zhou, H. Kim, T. Xiong,  X. W. Sun,   Every \Prufer\ ring does not have small finitistic dimension at most one, Comm. Algebra 48(12) (2020) 5311-5320.


\bibitem{x21}  J. Xie, Some properties of LPVDs and $t$-LPVDs. J. Algebra Appl,  to appear. https://doi.org/10.1142/S0219498822501134.

\bibitem{zhou17} D. C. Zhou, The theory of $w$-GD extension of domains, Ph. D. Dissertation, Sichuan Normal Univetsity, Chengdu, 2017.

\end{thebibliography}
\end{document}